\documentclass[10pt,a4paper]{article}
\usepackage[utf8]{inputenc}
\usepackage{amsmath}
\usepackage{amsfonts}
\usepackage{amssymb}
\usepackage{amsthm}
\usepackage{makeidx}
\usepackage{comment}
\usepackage{graphicx,color}
\usepackage{authblk}
\usepackage{subeqnarray}
\usepackage[pagewise]{lineno}
\usepackage{hyperref}
\usepackage[top=2.7cm, bottom=2.7cm, left=2.2cm, right=2.2cm]{geometry}

\author[1]{Jean Carlos Nakasato\thanks{nakasato@ime.usp.br}}
\author[1]{Marcone Corr\^ea Pereira\thanks{marcone@ime.usp.br}}
\affil[1]{Depto. Matem\'atica Aplicada, IME, Universidade de S\~ao Paulo,
Rua do Mat\~ao 1010, S\~ao Paulo - SP, Brazil}

\date{}
\title{The $p$-Laplacian in oscillating thin domains}
\newtheorem{teo}{Theorem}[section]
\newtheorem{lema}[teo]{Lemma}
\newtheorem{prop}[teo]{Proposition}
\newtheorem{cor}[teo]{Corollary}

\newtheorem{obs}{Remark}

\begin{document}
	\maketitle

	\begin{abstract}
		In this paper we study the asymptotic behavior of the solutions of the $p$-Laplacian equation posed in a 2-dimensional domain that degenerates into a line segment when a positive parameter $\varepsilon$ goes to zero (a thin domain perturbation). Also, we notice that high oscillatory behavior on the upper boundary of the thin domain is allowed as $\varepsilon \to 0$. Combining methods from classic homogenization theory and monotone operators we obtain the homogenized equation proving convergence of the solutions and establishing a corrector function which guarantees strong convergence in $W^{1,p}$ for $1<p<+\infty$.
	\end{abstract}


	\section{Introduction}
	In this work, we are interested in analyzing the asymptotic behavior of solutions of a nonlinear elliptic problem posed in a thin domain with high oscillating behavior on its boundary.
	
	In order to state the problem, let $g:\mathbb{R}\rightarrow \mathbb{R}$ be a function of class $\mathcal{C}^1$, $L$-periodic, positive with $0<g_0\leq g(x)\leq g_1$ for all $x\in \mathbb{R}$, where 
	$$
	g_0=\displaystyle\min_{x\in\mathbb{R}}g(x) 
	\quad \textrm{ and } \quad 
	g_1=\displaystyle\max_{x\in\mathbb{R}}g(x).
	$$ 
	
	Consider the bounded open set defined by
	\begin{equation*}
	R^{\varepsilon}=\left\{(x,y)\in\mathbb{R}^2:x\in(0,1)\mbox{ and }0<y<\varepsilon g(x/\varepsilon)\right\}
	\end{equation*}
	with $\varepsilon>0$ arbitrary. Notice that $R^{\varepsilon} \subset (0,1) \times (0,\varepsilon g_1)$ for any $\varepsilon>0$, and then, it sets a 2-dimensional thin domain as $\varepsilon \to 0$ since in some sense converges to the unit interval $(0,1) \subset \mathbb{R}$.
	
	In $R^{\varepsilon}$, we consider the following nonlinear elliptic problem with Neumann boundary condition 
	\begin{equation}\label{problem}
	\left\{ \begin{array}{ll}
	-\Delta_p w^{\varepsilon}+|w^{\varepsilon}|^{p-2}w^{\varepsilon}=h^\varepsilon\mbox{ in }R^\varepsilon,\\
	|\nabla w^{\varepsilon}|^{p-2}\nabla w^{\varepsilon} \cdot \nu^\varepsilon=0\mbox{ on }\partial R^\varepsilon,
	\end{array}\right.
	\end{equation}
	where $\nu^\varepsilon$ is the unit outward normal to $\partial R^\varepsilon$, $h^\varepsilon \in L^{p'}(R^\varepsilon)$, $1<p<+\infty$ with 
	$p^{-1} + {p'}^{-1} = 1$, and
	\begin{equation*}
	\Delta_p w^\varepsilon :=\partial_x\left(\left|\nabla w^\varepsilon\right|^{p-2}\partial_x w^\varepsilon\right)+\partial_y\left(\left|\nabla w^\varepsilon\right|^{p-2}\partial_yw^\varepsilon\right)
	\end{equation*} 
	denotes the  $p$-Laplacian operator. 
	We call \eqref{problem} the $p$-Laplacian equation with Neumann boundary condition. 
	It follows from Minty-Browder Theorem that it has an unique solution for each $\varepsilon>0$.
	
	Here we perform the asymptotic analysis of the problem \eqref{problem}.
	We obtain a homogenized equation to \eqref{problem} analyzing the convergence of the solutions as $\varepsilon$ goes to zero showing as the oscillating thin domain affects this quasilinear equation.
	
	In order to do that, we first perform the change of variables 
	$(x_1, x_2) = (x, y/\varepsilon)$
	which transforms the domain $R^\varepsilon$ into
	\begin{equation*}
	\Omega^\varepsilon=\left\{(x_1,x_2)\in \mathbb{R}^2: 0<x_1<1\mbox{ and } 0<x_2<g(x_1/\varepsilon)\right\}.
	\end{equation*}
	By doing so, we do not have a thin domain any more, even if it presents a high oscillatory behavior. 
	Indeed, in the oscillating domain $\Omega^\varepsilon$, we now consider the following problem
	\begin{equation}\label{problemchanged}
	\left\{ \begin{array}{ll}
	\displaystyle -\Delta_p^{\varepsilon^2} u_\varepsilon+|u_\varepsilon|^{p-2}u_\varepsilon=f^\varepsilon\mbox{ in }\Omega^\varepsilon,\\
	\displaystyle |\nabla^\varepsilon u_\varepsilon|^{p-2}\partial_{x_1}u_\varepsilon N_1^\varepsilon+\frac{1}{\varepsilon^2}|\nabla^\varepsilon u_\varepsilon|^{p-2}\partial_{x_2}u_\varepsilon N_2^\varepsilon=0\mbox{ on }\partial\Omega^\varepsilon,
	\end{array}\right.
	\end{equation}
	where  
	\begin{equation}\label{plaplaceepsilon}
	\begin{gathered}
	\Delta_p^{\varepsilon^2}u_\varepsilon:=\partial_{x_1}\left(|\nabla^\varepsilon u_\varepsilon|^{p-2}\partial{x_1}u_\varepsilon\right)+\frac{1}{\varepsilon^2}\partial_{x_2}\left(|\nabla^\varepsilon u_\varepsilon|^{p-2}\partial_{x_2}u_\varepsilon\right), \\
	\nabla^\varepsilon \cdot = (\partial_{x_1} \cdot, \varepsilon^{-1}\partial_{x_2} \cdot),  
	\end{gathered}
	\end{equation}
	and $N^\varepsilon=(N_1^\varepsilon,N_2^\varepsilon)$ is the unit outward normal to $\partial\Omega^\varepsilon$. 	

	It is not difficult to see that problems \eqref{problem} and \eqref{problemchanged} are equivalent. 
	Also, we notice that the variational formulation to \eqref{problemchanged} is the following one 
	\begin{equation*}
	\int_{\Omega^\varepsilon}\left\{ \left|\nabla^\varepsilon u_\varepsilon\right|^{p-2}\nabla^\varepsilon u_\varepsilon \nabla^\varepsilon \varphi+|u_\varepsilon|^{p-2}u_\varepsilon\varphi \right\} dx_1dx_2 = \int_{\Omega^\varepsilon} f^\varepsilon\varphi \, dx_1 dx_2 
	\end{equation*}	
	for any $\varphi \in W^{1,p}(\Omega^\varepsilon)$ with $\nabla^\varepsilon \cdot$ set in \eqref{plaplaceepsilon}.
	
	In our analysis, we take forcing terms $f^\varepsilon\in L^{p'}(\Omega^\varepsilon)$ uniformly bounded in $\varepsilon$.
	Indeed, we assume that the sequence $\hat f^\varepsilon \in L^{p'}(0,1)$ defined by
	\begin{equation} \label{fhat}
	\hat f^\varepsilon(x_1) = \int_0^{g(x_1/\varepsilon)} f^\varepsilon(x_1,x_2) \, dx_2
	\quad\mbox{ satisfies }\quad
	\hat f^\varepsilon \rightharpoonup \hat f \textrm{ weakly in } L^{p'}(0,1), \quad \textrm{ as } \varepsilon \to 0, 
	\end{equation}
	for some function $\hat f \in L^{p'}(0,1)$.  We point out that $h^{\varepsilon}$ and $f^{\varepsilon}$ are related with a simple change of variables, that is, $f^{\varepsilon}(x_{1},x_{2})=h(x_{1},\varepsilon x_{2})$. As an example for the forcing term, one can take $f^\varepsilon(x_1,x_2)=f(x_1)$, $f\in L^{p'}(0,1)$.  

	Now observe that the coefficient ${1}/{\varepsilon^2}$ in front of the second term of \eqref{plaplaceepsilon} corresponds to a high diffusion mechanism in the $x_2$-direction as $\varepsilon \to 0$. 
	Indeed, because of this very strong diffusion mechanism, we expect that the solutions become homogeneous in the $x_2$-direction as $\varepsilon$ goes to zero. Thus, the limiting solution of the problem will not get dependence on the $x_2$-variable, and then, the limiting problem will be 1-dimensional which is in agreement with the intuitive idea that an equation in a thin domain should approach an equation in a line segment.
	
	Since we are considering here a boundary perturbation problem, we need an approach to deal with functions whose domain varies. In fact, we need to set a notion of convergence in order to establish our homogenized equation.

	In a certain way, we have that the oscillating domain $\Omega^\varepsilon$ fills the entire rectangle 
	$$\Omega:=(0,1)\times(0,g_1)$$ 
	at $\varepsilon=0$. 
	Hence, we can expect that the solutions of \eqref{problemchanged} should converge to the solutions of a limit equation sets in $\Omega$. 
	Then, since the limit solution will not depend on the variable $x_2$, we will obtain a 1-dimensional equation as result.
	
	Thus, we need to compare functions defined in $\Omega^\varepsilon$ with functions set in the rectangle $\Omega$. 
	Here, we will use the extension operator approach deeply applied in homogenization theory \cite{cioranescu2012, palencia1980non, Tartar2009}. 
	We will get a bounded operator $P_\varepsilon$ which transforms functions defined in $\Omega^\varepsilon$ in functions set in $\Omega$.
	In fact, we will get convergence in $W^{1,p}(\Omega)$ using the extension operator $P_\varepsilon:W^{1,p}(\Omega^\varepsilon)\rightarrow W^{1,p}(\Omega)$ introduced in \cite{arrieta2011} (see Lemma \ref{lema3.1} below).
	
	We show that the solutions $u_\varepsilon \in W^{1,p}(\Omega^\varepsilon)$ of \eqref{problemchanged} satisfy   
	$$
	P_\varepsilon u_\varepsilon \rightharpoonup u_0\mbox{ weakly in } W^{1,p}(\Omega), \quad \textrm{ as } \varepsilon \to 0, 
	$$
	where $u_0(x_1,x_2)=u_0(x_1)$ is the unique solution of the 1-dimensional equation 
	\begin{eqnarray} \label{limitintro} 
	\left\lbrace\begin{array}{ll}
	-q\left(|u'_0|^{p-2}u'_0\right)'+|u_0|^{p-2}u_0=\bar{f} \quad \mbox{ in }(0,1),\\
	u'_0(0)=u'_0(1)=0,
	\end{array}
	\right.
	\end{eqnarray}
	where the set $Y^*$ denotes the representative cell of the oscillating domain $\Omega^\epsilon$
	\begin{equation*}  
	Y^*=\left\lbrace \right(y_1,y_2) \in \mathbb{R}^2 : 0<y_1<L\mbox{ and }0<y_2<g(y_1)\rbrace.
	\end{equation*} 
	and $\bar f \in L^{p'}(0,1)$ is given by 
	$$
	\bar f = \frac{L}{|Y^*|} \hat f
	$$
	with $\hat f$ introduced in \eqref{fhat}. The homogenized coefficient $q$ is the constant defined by  
	\begin{equation*} 
	q=\frac{1}{|Y^*|}\int_{Y^*} |\nabla v|^{p-2}\partial_{y_1}v \, dy_1dy_2
	\end{equation*}
	where function $v$ is the unique solution of the auxiliar problem:
		\begin{equation}\label{auxint}
		\begin{gathered}
		\int_{Y^*} |\nabla v|^{p-2}\nabla v\nabla \varphi=0\mbox{, } \quad \forall \varphi\in W^{1,p}_{per}(Y^*) \\
		\textrm{ with } \qquad (v - \, y_1) \in W^{1,p}_{per}(Y^*).
		\end{gathered}
		\end{equation}

	The Banach space $W_{per}^{1,p}(Y^*)$ denotes 
	the functions in $W^{1,p}(Y^*)$ whose trace on the lateral faces of $Y^*$ are equal and possess average zero. 
	That is, if $\partial_{left} Y^*$ and $\partial_{right} Y^*$ are respectively the left and right side of the boundary $\partial Y^*$ and 
	$$< \varphi>_{\mathcal{O}} := \frac{1}{|\mathcal{O}|} \int_{\mathcal{O}} \varphi(x) \, dx$$
	is the average of $\varphi \in L^1_{loc}(\mathbb{R}^2)$ on an open bounded set $\mathcal{O} \subset \mathbb{R}^2$, we take
	$$
	W_{per}^{1,p}(Y^*) = \{ \varphi\in W^{1,p}(Y^*) \, : \,  \varphi |_{\partial_{left} Y^*} = \varphi|_{\partial_{right} Y^*} \textrm{ with } < \varphi>_{Y^*} = 0 \}.
	$$
	Also, $|\mathcal{O}|$ denotes the Lebesgue measure of any measurable set $\mathcal{O} \subset \mathbb{R}^2$.
	
	Note that the functions $\varphi \in W_{per}^{1,p}(Y^*)$ can be periodically extended to the horizontal direction, in such way that 
	$\varphi(y_1+L,y_2) = \varphi(y_1,y_2)$ for all $y_1 \in \mathbb{R}$ and $0<y_2<g(y_1)$. 
	The existence and uniqueness of the solutions of \eqref{auxint} is also a consequence of Minty-Browder Theorem.
	

	
	To accomplish our goal, we use techniques from \cite{arrieta2011} and \cite{donato1990}. In \cite{arrieta2011}, the authors have considered this same singular boundary perturbation problem for the Laplacian operator ($p=2$); in \cite{donato1990}, a monotone operator in a periodically perforated domain for a class of operators such that the $p$-Laplacian equation fits in is studied. Here, we combine these techniques to set appropriated test functions to identify the homogenized limit and show convergence. We rigorously derive an effective $1$-dimensional model as $\varepsilon$ goes to zero. Moreover, using the corrector approach discussed in \cite{MasoA}, we construct a family of correctors which allow us to obtain strong convergence in $W^{1,p}(\Omega^\varepsilon)$.

	We observe that the same issues can be considered to oscillating thin domains in $\mathbb{R}^{N+1}$ with $N\geq2$. The same arguments can be performed to $R^\varepsilon$ defined by a positive and periodic function $g: \omega \mapsto \mathbb{R}$ where $\omega \subset \mathbb{R}^{N}$ is a cube as $\omega = (0,L_1) \times ... \times (0,L_N)$. We assume $N=1$ just to simplify the notations and proofs.

	In the literature one can find several works concerned with partial differential equations posed in thin domains. Indeed, it is not difficult to realize that they can occur in many applications. For instance, they can be found in mathematical models for ocean dynamics (where one is dealing with fluid regions which are thin compared to the horizontal length scales), lubrication, nanotechnology, blood circulation, material engineering, meteorology, etc. 
	Many techniques and methods have been developed in order to understand the effect of the geometry and thickness of the domain on the solutions of such singular problems.
	
	From pioneering works to recent ones we mention  \cite{SM, HaleRaugel, PR, EI, FMP, MPov, AM, AM2} concerned with elliptic and parabolic equations, as well as \cite{A, DI, BFN, FKTW, XL, MVel, BPazanin, MGrau} where the authors considered Stokes and Navier-Stokes equations from fluid mechanics. 
	The second author also have studied different classes of thin domains problems for elliptic and parabolic equations.
	For instance we mention the recent works \cite{AP10, MAMPA, MRi, SMarcone, MZAMP}. See also \cite{GH, MJ, MJ2} where nonlocal equations in thin structures have been considered. Recently, we also studied the same problem and proved the results here presented with a different approach, see \cite{nakasato1}.
	
	For monotone operators in standard thin domains, that is, those ones without oscillating boundary, we mention the recent works \cite{MRi2, Ri} where thin channels in $\mathbb{R}^N$ where considered. 
	

	Finally we notice that different conditions on the lateral boundaries of the thin domain may be set preserving the Neumann type boundary condition on upper and lower boundary of $R^\varepsilon$. Dirichlet or even Robin homogeneous can be set, and then, the limit problem will preserve this boundary condition as a point condition. On the other hand, as we know from \cite{MRi}, if we suppose Dirichlet boundary condition in whole $\partial R^\varepsilon$, the family of solutions will converge to the null function as the parameter $\varepsilon \to 0$.
	
	The paper is organized as follows: in Section 2, we collect some basic facts to monotone operators and introduce the extension operator $P_\varepsilon$. In Section 3, we prove our main result concerned to the convergence of the solutions and the homogenized equation. Finally, we obtain a corrector result in Section 4.
	
	\section{Preliminary Results}
	
	Here, we recall some results that will be useful in the next sections. 
	We start with some ones concerned to the $p$-laplacian operator (see \cite{lindqvist}).
	
	\begin{prop}\label{proposicaoplaplaciano}
		Let $x,y\in\mathbb{R}^n$.
		\begin{itemize}
			\item If $p> 2$, then
			\begin{equation*} 
			<|x|^{p-2}x-|y|^{p-2}y,x-y> \geq c_p|x-y|^p.
			\end{equation*}
			\item If $1<p<2$, then
			\begin{eqnarray*} 
			<|x|^{p-2}x-|y|^{p-2}y,x-y>\geq c_p|x-y|^2(|x|+|y|)^{p-2} \geq c_p|x-y|^2(1+|x|+|y|)^{p-2}.
			\end{eqnarray*}
		\end{itemize}
	\end{prop}

	\begin{cor}\label{corolarioplaplaciano}
		Let $a_p:\mathbb{R}^n\rightarrow \mathbb{R}^n$ defined by $a_p(s)=|s|^{p-2}s$ and $p' > 1$ such that $\frac{1}{p}+\frac{1}{p'}=1$. Then, $a_{p}$ is the inverse of $a_{p'}$. Moreover,
		\begin{itemize}
			\item If $1<p'< 2$ (i.e, $p\geq 2$), then
			\begin{equation*}
			\left||u|^{p'-2}u-|v|^{p'-2}v\right|\leq c|u-v|^{p'-1}.
			\end{equation*}
			\item If $p'\geq 2$ (i.e, $1<p\leq 2$), then
			\begin{eqnarray*}
			\left||u|^{p'-2}u-|v|^{p'-2}v\right|\leq c|u-v|(|u|+|v|)^{p'-2}\leq c|u-v|(1+|u|+|v|)^{p'-2}.
			\end{eqnarray*}
		\end{itemize}
	\end{cor}

	Now let us introduce a lemma concerning the existence of an extension operator which will be used to transform functions defined in $\Omega^\varepsilon$ into functions given in the fixed domain $\Omega$. This lemma will be very important in the proof our main result. 
	
	\begin{lema}[Extension Operator]\label{lema3.1}
		Let 
		\begin{eqnarray*}
			\mathcal{O}&=&\left\{(x_1,x_2)\in \mathbb{R}^2:x_1\in I\mbox{ and }0<x_2<G_1\right\}\\
			\mathcal{O}^\varepsilon&=& \left\{(x_1,x_2)\in \mathbb{R}^2:x_1\in I\mbox{ and }0<x_2<G_\varepsilon(x_1)\right\},
		\end{eqnarray*}
		where $I\subset \mathbb{R}$ is an open interval, $G_\varepsilon:I\rightarrow\mathbb{R}$ is a $C^1$-function satisfying 
		$$0<G_0\leq G_\varepsilon(x_{1})\leq G_1 \textrm{ for all } x \in I \textrm{ an } \varepsilon>0.$$ 
		Then, there exists an extension operator
		\begin{equation*}
		P_\varepsilon\in\mathcal{L}\left(W^{1,p}(\mathcal{O}^\varepsilon),W^{1,p}(\mathcal{O})\right)\cap\mathcal{L}\left(L^p(\mathcal{O}^\varepsilon),L^p(\mathcal{O})\right)\cap \mathcal{L}\left(W^{1,p}_{\partial_l}(\mathcal{O}^\varepsilon),W^{1,p}_{\partial_l}(\mathcal{O})\right)
		\end{equation*}
		where $W^{1,p}_{\partial_l}$ is is the set of functions in $W^{1,p}$ which are zero on the lateral boundaries. 
		
		Moreover, there exists a constant $K$ independent of $\varepsilon$ a $p$ such that  
		\begin{eqnarray*}
			\left\{\begin{array}{lll}
				\left|\left|P_\varepsilon\varphi\right|\right|_{L^p(\mathcal{O})}\leq K\left|\left|\varphi\right|\right|_{L^p(\mathcal{O}^\varepsilon)}\\
				\left|\left|\partial_{x_1}P_\varepsilon\varphi\right|\right|_{L^p(\mathcal{O})}\leq K\left(\left|\left|\partial_{x_1}\varphi\right|\right|_{L^p(\mathcal{O}^\varepsilon)}+\eta(\varepsilon)\left|\left|\partial_{x_2}\varphi\right|\right|{}_{L^p(\mathcal{O}^\varepsilon)}\right)\\
				\left|\left|\partial_{x_2}P_\varepsilon\varphi\right|\right|_{L^p(\mathcal{O})}\leq K\left|\left|\partial_{x_2}\varphi\right|\right|_{L^p(\mathcal{O}^\varepsilon)}
			\end{array} \right.
		\end{eqnarray*}
		for all function $\varphi\in W^{1,p}(\mathcal{O}_\varepsilon)$ with $1\leq p\leq\infty$, and $\eta$ given by 
		\begin{equation*}
		\eta(\varepsilon)=\sup_{x\in I}\left(\left|G_\varepsilon'(x)\right|\right).
		\end{equation*}
	\end{lema}
	\begin{proof}
		For a proof see \cite{arrieta2011} or \cite{AP10}.
	\end{proof}
	
	\begin{obs} 
		$(i)$  This operator preserves periodicity in the first variable: if the function $\varphi_\varepsilon(x_1,x_2)$ is periodic in $x_1$, then 
		the extended function $P_\varepsilon \varphi_\varepsilon$ is also periodic in $x_1$. 
		
		$(ii)$ We also can use this lemma to the case $G_\varepsilon(x_1)=G(x_1)$ independent of $\varepsilon$. In particular, we can
		apply the extension operator to the basic cell $Y^*$. 
	\end{obs}
	
	\section{Convergence theorem}

	Now we are in condition to show the main result of the paper. 
	\begin{teo}\label{maintheorem}
		Let $u_\varepsilon$ be the sequence of solutions of problem \eqref{problemchanged} 
		with $f^\varepsilon \in L^{p'}(\Omega^\varepsilon)$ satisfying 
		$\| f^\varepsilon \|_{L^{p'}(\Omega^\varepsilon)} \le C$ for some positive constant $C$ independent of $\varepsilon > 0$. 
		Assume that the function 
		$$
		\hat f^\varepsilon(x_1)=\int_0^{g(x_1/\varepsilon)} f^\varepsilon(x_1,x_{2})dx_{2}
		$$ 
		satisfies $\hat f^\varepsilon \rightharpoonup \hat f$, weakly in $L^{p'}(0,1)$, as $\varepsilon \to 0$. 
		
		Then, there exist a function $u_0(x_1,x_2) = u_0(x_1) \in W^{1,p}(0,1)$ and an extension operator $P_\varepsilon:W^{1,p}(\Omega^\varepsilon)\rightarrow W^{1,p}(\Omega)$ such that 
		$$
		P_\varepsilon u_\varepsilon\rightharpoonup u_0, \quad \textrm{ weakly in } W^{1,p}(\Omega). 
		$$

		Moreover, we have that the function $u_0$ is the solution of the one dimensional $p$-Laplacian problem with constant coefficient 
		\begin{equation} \label{limitV}
		q\int_0^1 \,|u'_0|^{p-2}u'_0  \psi' dx_1+\int_{0}^1 |u_0|^{p-2}u_0 \psi dx_1=\int_0^1\overline{f}\psi dx_1,
		\qquad \forall \psi\in W^{1,p}(0,1),
		\end{equation}
		where
		\begin{equation}\label{defq}
		q=\frac{1}{|Y^*|}\int_{Y^*}\left|\nabla v\right|^{p-2}\partial_{y_1}v \, dy_{1}dy_2,
		\end{equation}
		and $\bar f \in L^{p'}(0,1)$ is the forcing term given by 
		$$
		\bar f = \frac{L}{|Y^*|} \hat f.
		$$
		The set $Y^*$ is the representative cell of $\Omega^\varepsilon$ 	
		and the function $v$ is the unique solution of the auxiliary problem \eqref{auxint}.

	\end{teo}

		\begin{obs}
		It is worth noting that our result also sets the homogenized problem to the Laplacian operator in oscillating thin domains as has been done in \cite[Theorem 4.3]{arrieta2011}. We have just to take $p=2$ in Theorem \ref{maintheorem}.
	\end{obs}

		\begin{obs}
 		We still emphasize that the limit equation \eqref{limitV}  is the one-dimensional $p$-Laplacian equation with constant coefficient $q$ and we also note that it is well defined. Indeed, it follows from \eqref{auxint} that the homogenized coefficient $q$ is positive. If we take $\varphi = v-y_1$ as a test function in \eqref{auxint}, since $v \neq 0$ in $L^p(Y^*)$, we get 
		$$
		q |Y^*| = \int_{Y^*} |\nabla v|^p > 0.
		$$
		 
		\end{obs}

	\begin{proof}

	        The proof of Theorem \ref{maintheorem} will consist of three steps.
		First, we show that the solutions are uniformly bounded obtaining convergent subsequences. 
		Next we introduce an appropriated auxiliar partition to $\Omega^\varepsilon$ to identify the homogenized equation and define auxiliar functions which will help us to pass to the limit. 
		Finally, we pass to the limit in the problem obtaining the desired result.  
		
\vspace{.2cm}

		\textbf{Uniform bounds of solutions}
		
		Let us recall that the variational formulation of problem \eqref{problemchanged} is
		\begin{equation}\label{variacionalproblemchanged}
		\int_{\Omega^\varepsilon}\left|\nabla^\varepsilon u_\varepsilon\right|^{p-2}\nabla^\varepsilon u_\varepsilon \nabla^\varepsilon \varphi+|u_\varepsilon|^{p-2}u_\varepsilon\varphi= \int_{\Omega^\varepsilon} f^\epsilon \, \varphi, \quad \forall \varphi\in W^{1,p}(\Omega^\varepsilon)
		\end{equation}
		with $\nabla^\varepsilon \cdot = (\partial_{x_1} \cdot, \varepsilon^{-1}\partial_{x_2} \cdot)$. 
		Hence, if $u_\varepsilon$ is the solution of \eqref{problemchanged}, and we take $\varphi=u_\varepsilon$ in \eqref{variacionalproblemchanged} to obtain  
		\begin{eqnarray*}
			||u_\varepsilon||_{W^{1,p}(\Omega^\varepsilon)}^p\leq \int_{\Omega^\varepsilon}|\nabla^\varepsilon u_\varepsilon|^p+|u_\varepsilon|^p\leq ||f^\varepsilon||_{L^{p'}(\Omega^\varepsilon)} ||u_\varepsilon||_{W^{1,p}(\Omega^\varepsilon)}, \quad \forall \varepsilon \in (0,1).
		\end{eqnarray*}
		Thus, $u_\varepsilon$ is uniformly limited in $W^{1,p}(\Omega^\varepsilon)$, and then, there exists $c>0$, independent of $\varepsilon$, such that 
		\begin{eqnarray*} 
		\left|\left|\partial_{x_1} u_\varepsilon\right|\right|_{L^p(\Omega^\varepsilon)} \leq c \quad \textrm{ and } \quad 
		\frac{1}{\varepsilon}\left|\left|\partial_{x_2} u_\varepsilon\right|\right|_{L^p(\Omega^\varepsilon)} \leq c.\label{limitacaouepsilonsegundavariavel}
		\end{eqnarray*}
		
		Now, let $P_\varepsilon$ be the extension operator given by Lemma \ref{lema3.1}. 
		In order to obtain the homogenized equation, we rewrite problem \eqref{variacionalproblemchanged} to the fixed domain $\Omega=(0,1)\times(0,g_1)$ in the following way
		\begin{equation}\label{extendedvariationalproblem}
		\int_\Omega\left(\widetilde{a_{p}(\nabla^{\varepsilon}u_{\varepsilon})}\nabla^\varepsilon \varphi+\chi_{\Omega^\varepsilon} \left|P_\varepsilon u_\varepsilon\right|^{p-2} P_\varepsilon u_\varepsilon \varphi \right)=\int_\Omega \chi_{\Omega^\varepsilon} f^\varepsilon\varphi
		\end{equation}
		for all $\varphi\in W^{1,p}(\Omega)$ where $a_p(s)=|s|^{p-2}s$ is the function introduced by Corollary \ref{corolarioplaplaciano}, $\widetilde{\cdot}$ denotes the extension by zero and $\chi_{\Omega^\varepsilon}$ is the characteristic function of $\Omega^\varepsilon$. 
		
		Since $u_\varepsilon$ is uniformly bounded, it follows from Lemma \ref{lema3.1} that
		\begin{eqnarray}\label{limitacaopuepsilon}
		\left\{\begin{array}{lll}
		\left|\left|P_\varepsilon u_\varepsilon\right|\right|_{L^{p}(\Omega)}\leq c,\\
		\left|\left|\partial_{x_1}P_\varepsilon u_\varepsilon\right|\right|_{L^{p}(\Omega)}\leq c, \\
		\left|\left|\partial_{x_2}P_\varepsilon u_\varepsilon\right|\right|_{L^{p}(\Omega)}\leq c \varepsilon,
		\end{array}
		\right.
		\end{eqnarray}
		for some positive constant $c$ independent of $\varepsilon$. 
		Also, we have that
		\begin{equation}\label{limitacaooperador}
		\left|\left|\widetilde{\left.a_{p}(\nabla^{\varepsilon}u_{\varepsilon})\right.}\right|\right|_{\left(L^{p'}(\Omega)\right)^2}
		\leq c.
		\end{equation}
		
		Therefore, from \eqref{limitacaopuepsilon} and \eqref{limitacaooperador}, 
		we can extract a subsequence, still denoted in the same way, such that,
		for some functions $u_0 \in W^{1,p}(\Omega)$ and $a_0 \in L^{p'}(\Omega)\times L^{p'}(\Omega)$, we have  
		\begin{eqnarray} \label{901}
		\left\{\begin{array}{lll}
		P_\varepsilon u_\varepsilon \rightharpoonup u_0\mbox{ weakly in }W^{1,p}(\Omega),\\
		P_\varepsilon u_\varepsilon \rightarrow u_0\mbox{ strongly in }L^p(\Omega ),\\
		\widetilde{a_{p}(\nabla^{\varepsilon}u_{\varepsilon})}\rightharpoonup a_0\mbox{ weakly in }L^{p'}(\Omega)\times L^{p'}(\Omega)\label{conva00}.
		\end{array}
		\right.
		\end{eqnarray}
		
		Indeed, due to \eqref{limitacaopuepsilon} and \eqref{conva00}, we obtain that $u_0(x_1,x_2)=u(x_1)$. In fact,
		\begin{eqnarray*}
			\int_{\Omega}u_0\partial_{x_2}\varphi dx_1dx_2=\lim_{\varepsilon\rightarrow 0}\int_\Omega P_\varepsilon u_\varepsilon\partial_{x_2}\varphi dx_1dx_2
			=-\lim_{\varepsilon\rightarrow 0} \int_\Omega\partial_{x_2} P_\varepsilon u_\varepsilon\varphi=0
		\end{eqnarray*}
		for all $\varphi\in C_0^\infty(\overline{\Omega})$. Then,
		\begin{equation} \label{dnula}
		\partial_{x_2}u_0=0\mbox{ a.e. in }\Omega
		\end{equation}
		and $u_0 \in W^{1,p}(0,1)$.
		
		Next, we show that 
		\begin{equation} \label{eq870}
		\chi_{\Omega^\varepsilon}\stackrel{*}{\rightharpoonup}\theta \quad \mbox{ weakly star in }L^\infty(\Omega), 
		\end{equation}
		where 
		$$
		\theta(x_2) := \frac{1}{L} \int_0^L \chi_{Y^*}(s,x_2) ds \textrm{ a.e. } x_2 \in (0,g_1) 
		$$
		and $\chi_{Y^*}$ is the characteristic function of $Y^*$.
		Also, we note that $\theta$ satisfies
		\begin{equation} \label{Y^*}
		L\int_0^{g_1}\theta(x_2)dx_2=|Y^*|.
		\end{equation}
		
		In fact, if we extend $\chi_{Y^*}$ periodically to the horizontal direction $x_1$, we get that
		\begin{equation} \label{chip}
		\chi_{\Omega^\varepsilon}(x_1,x_2) 
		= \chi_{Y^*}\left(\frac{x_1}{\epsilon},x_2\right) \textrm{ in } \Omega.
		\end{equation}
		Then, due to \eqref{chip}, we have that 
		\begin{equation} \label{chi0}
		\chi_{\Omega^\varepsilon}( \cdot, x_2) \stackrel{\varepsilon\to 0}{\rightharpoonup} \theta(x_2)  
		\quad \textrm{ weakly star in } L^\infty(0,1), 
		\end{equation}
		for all $x_2\in (0,g_1)$. Hence, from (\ref{chi0}), we have that 
		$$
		H^\varepsilon(x_2) := \int_0^1  \varphi(x_1,x_2) \, \Big\{ \chi_{\Omega^\varepsilon}(x_1,x_2) - \theta(x_2) \Big\} \, dx_1 
		\to 0 
		$$ 
		as $\varepsilon \to 0$, a.e. $x_2 \in (0, g_1)$, and for all $\varphi \in L^1(\Omega)$.
		Thus, due to
		$$
		\begin{gathered}
		\int_{\Omega} \varphi(x_1,x_2) \,\Big\{ \chi_{\Omega^\varepsilon}(x_1,x_2) - \theta(x_2) \Big\} \, dx_1 dx_2 
		= \int_0^{g_1} H^\varepsilon_i(x_2) dx_2 \\ 
		\textrm{ and }  |H^\varepsilon(x_2)| \le \int_0^1 | \varphi(x_1,x_2)| dx_1,
		\end{gathered}
		$$
		we get \eqref{eq870} from Lebesgue's Dominated Convergence Theorem.

		Now, using \eqref{conva00}, Corollary \ref{corolarioplaplaciano} and a H\"older's inequality, one can conclude that
		\begin{eqnarray*}
		\int_\Omega \left|\left|P_\varepsilon u_\varepsilon\right|^{p-2}P_\varepsilon u_\varepsilon -\left|u_0\right|^{p-2}u_0\right|^{p'}dx_1dx_2\rightarrow 0.
		\end{eqnarray*}
		Thus, we can conclude that
		\begin{equation}\label{Pepsilonuepsilon}
		\chi_{\Omega^\varepsilon} \left|P_\varepsilon u_\varepsilon\right|^{p-2} P_\varepsilon u_\varepsilon \rightharpoonup \theta |u_0|^{p-2}u_0
		\mbox{ weakly in }  L^{p'}(\Omega).
		\end{equation}

		Notice that we can pass to the limit in \eqref{extendedvariationalproblem} taking test functions depending just on the first variable, 
		that is, taking $\varphi(x_1,x_2) = \varphi(x_1) \in W^{1,p}(0,1)$ in \eqref{extendedvariationalproblem}. 
		Indeed, we obtain from \eqref{extendedvariationalproblem}, \eqref{fhat}, \eqref{901}, \eqref{dnula} and \eqref{Pepsilonuepsilon} that  
		\begin{eqnarray} \label{pseudolimit}
		\int_\Omega \left\{  a_0(x_1,x_2) \cdot (\partial_{x_1} \varphi(x_1) , 0) + \theta(x_2) |u_0(x_1)|^{p-2} u_0(x_1) \, \varphi(x_1)  \right\} dx_1 dx_2  
		 = \int_0^1 \varphi(x_1) \hat f(x_1) \, dx_1
		\end{eqnarray}
		for all $\varphi \in W^{1,p}(0,1)$, since by \eqref{fhat}, we have 
		\begin{eqnarray*}
			\int_\Omega \chi_{\Omega_\varepsilon} f^\varepsilon \varphi \, dx_1 dx_2 & = & \int_0^1 \varphi(x_1) \int_0^{g(x_1/\varepsilon)} f^\varepsilon(x_1,x_2) \, dx_2 dx_1
			 \to  \int_0^1 \varphi(x_1) \hat f(x_1) \, dx_1.
		\end{eqnarray*}
		
		Thus, we get from \eqref{Y^*} and \eqref{pseudolimit} that 
		\begin{equation}\label{prehomogenized}
		\begin{gathered} 
		\int_0^1 \left\{ \left( \int_0^{g_1} a_0 \cdot (1,0) \, dx_2 \right) \partial_{x_1} \varphi + \frac{|Y^*|}{L} |u_0|^{p-2} u_0 \varphi \right\} dx_1  = \int_0^1 \varphi(x_1) \hat f(x_1) \, dx_1
		\end{gathered}
		\end{equation}
		for all $\varphi \in W^{1,p}(0,1)$ where $(1,0)$ is the first vector of the canonical basis of $\mathbb{R}^2$.

		\begin{obs}
			Our goal now is to identify function 
			$$\int_0^{g_1} a_0 \cdot (1,0) \, dx_2 $$ 
			from the limit \eqref{conva00} and \eqref{prehomogenized}. We will show that  
			$$\int_0^{g_1} a_0 \cdot (1,0) \, dx_2  = \int_0^{g_1}b(\partial_{x_1} u_0, x_2)\cdot(1,0) \, dx_2$$ 
			where 
		$b:\mathbb{R}\times (0,g_1)\rightarrow \mathbb{R}^2$ is given by 
		\begin{equation} \label{homogenizedb}
		b(\xi,x_2)=  \frac{a_p(\xi)}{L}\int_0^L \chi_{Y^*}(s,x_2) \, a_p((1,0)+\nabla P\phi (s,x_2)) \, ds,
		\end{equation}
		$P$ is the extension operator given by Lemma \ref{lema3.1} to the open cell $Y^*$, $a_p(s) = |s|^{p-2}s$ and $\phi = v - y_1$ is the function set by the auxiliary problem \eqref{auxint}.
		
		Moreover, we will get the following convergence 
		\begin{equation} \label{lim01}
		\int_0^{g_1} \left|\widetilde{\nabla^\varepsilon u_\varepsilon}\right|^{p-2}\widetilde{\nabla^\varepsilon u_\varepsilon}  \cdot (1,0) dx_2
		\rightharpoonup \int_0^{g_1} b(\partial_{x_1} u_0,x_2) \cdot (1,0) dx_2
		\end{equation} 
		weakly in  $L^{p'}(0,1)$ where $ \tilde \cdot$ is the standard extension by zero and 
		$\nabla^\varepsilon \cdot = (\partial_{x_1} \cdot, \varepsilon^{-1}\partial_{x_2} \cdot)$.

			For this sake, we will proceed as in \cite{arrieta2011, donato1990, FM}. 
			We introduce an appropriated partition to $\Omega$, 
			as well as, some auxiliary sequences which will allow us to achieve this goal.
		\end{obs} 
		
				\vspace{.5cm}
		
		\noindent\textbf{Auxiliar Partition}

		For all $\nu\in\mathbb{N}$, we consider the partition $(A_{i\nu})$ of $\Omega$ in rectangles $A_{i\nu}$ such that its base has length $2^{-\nu}$ and height $g_1$. Then,
		\begin{equation} \label{Anu}
		{\rm int}\left( \bigcup_{i=0}^{2^\nu-1} \bar A_{i\nu} \right) =\Omega.
		\end{equation}

		Now, take a function $w_0\in W^{1,p}(0,1)$, and for each $i$, consider the average 
		$$<\partial_{x_1}w_0>_{i\nu} = \frac{1}{|A_{i\nu}|} \int_{A_{i\nu}} \partial_{x_1}w_0(x_1) \, dx_1 dx_2.$$ 
		Next, using the solution $v$ of \eqref{auxint}, set
		$$
		v_{i\nu}(y_1,y_2) = <\partial_{x_1}w_0>_{i\nu} \, v(y_1,y_2) 
		 =  <\partial_{x_1}w_0>_{i\nu} \left(y_1+\phi(y_1,y_2)\right) \quad \textrm{ in } Y^* 
		$$
		for some $\phi \in W^{1,p}_{per}(Y^*)$.

		By Lemma \ref{lema3.1}, there exists an extension operator $P$ such that $P\phi\in W^{1,p}_{per}(Y)$, where $Y=(0,L)\times(0,g_1)$. 
		Extend $P\phi$ periodically in the first variable and define
		\begin{equation*} 
		w_{i\nu}(y_1,y_2)=<\partial_{x_1}w_0>_{i\nu}\left(P\phi(y_1,y_2)+y_1\right)\mbox{, }\forall (y_1,y_2)\in\mathbb{R}\times (0,g_1).
		\end{equation*}	
		
		Consider the sequence 
		\begin{eqnarray}   \label{wepsi}
		w_{i\nu}^\varepsilon(x_1,x_2)=\varepsilon w_{i\nu}\left(\frac{x_1}{\varepsilon},x_2\right)
		=\varepsilon <\partial_{x_1}w_0>_{i\nu}P\phi\left(\frac{x_1}{\varepsilon},x_2\right)+< \partial_{x_1}w_0>_{i\nu}x_1,
		\end{eqnarray}
		where $(x_1,x_2)\in \Omega$.
		
		It is not difficult to prove that (see for instance \cite[page 5521, item (d)]{arrieta2011} for the proof)
		\begin{eqnarray}\label{winuconv}
		\left\lbrace \begin{array}{lll}
		w_{i\nu}^\varepsilon\rightarrow <\partial_{x_1}w_0>_{i\nu}x_1\mbox{ in }L^p(\Omega),\\
		\partial_{x_1}w_{i\nu}^\varepsilon\rightharpoonup <\partial_{x_1}w_0>_{i\nu}\mbox{ in }L^p(\Omega),\\
		\partial_{x_2}w_{i\nu}^\varepsilon\rightarrow 0 \mbox{ in } L^p(\Omega),
		\end{array}\right.
		\end{eqnarray}
		when $\varepsilon\rightarrow 0$.
		
		Now, consider $\nabla P \phi$ and define
		\begin{equation*}
		W_{i\nu}(y_1,y_2)= <\partial_{x_1}w_0>_{i\nu}\nabla P \phi(y_1,y_2)+\left(< \partial_{x_1}w_0>_{i\nu},0\right)
		\end{equation*}
		and
		\begin{eqnarray}\label{ainu}
		a_{i\nu}(y_1,y_2) & = & \chi_{Y^*}(y_1,y_2)\left|W_{i\nu}(y_1,y_2)\right|^{p-2}W_{i\nu}(y_1,y_2) \nonumber \\
		& = & \chi_{Y^*}(y_1,y_2)a_p\left(< \partial_{x_1}w_0>_{i\nu}\right)a_p\left(\nabla P\phi(y_1,y_2)+(1,0)\right),
		\end{eqnarray}
		where $(y_1,y_2)\in \mathbb{R}\times (0,g_1)$.
		Next, consider the sequences
		\begin{eqnarray*}
		W_{i\nu}^\varepsilon(x_1,x_2)= W_{i\nu}\left(\frac{x_1}{\varepsilon},x_2\right)
		=<\partial_{x_1}w_0>_{i\nu}\nabla P \phi\left(\frac{x_1}{\varepsilon},x_2\right)+\left(< \partial_{x_1}w_0>_{i\nu},0\right)
		\end{eqnarray*}
		and        
		\begin{equation}\label{aepsilon}
		\begin{gathered}
		a^\varepsilon_{i\nu}(x_1,x_2)=a_{i\nu}\left(\frac{x_1}{\varepsilon},x_2\right)=\chi_{\Omega^\varepsilon}(x_1,x_2)\left|W_{i\nu}^\varepsilon(x_1,x_2)\right|^{p-2}W_{i\nu}^\varepsilon(x_1,x_2)\\
		=a_p\left(< \partial_{x_1}w_0>_{i\nu}\right) \chi_{\Omega^\varepsilon}(x_1,x_2) a_p\left(\nabla P\phi \left(\frac{x_1}{\varepsilon},x_2\right)+(1,0)\right)
		\end{gathered}
		\end{equation}
		defined in  $\Omega$.

		Remember that $\nabla^\varepsilon = (\partial_{x_1}\cdot,\varepsilon^{-1}\partial_{x_2}\cdot)$. Thus, due to \eqref{wepsi}, we get 
		\begin{equation*}
		\nabla w_{i\nu}^\varepsilon(x_1,x_2)=\left(\partial_{x_1}w_{i\nu}\left(\frac{x_1}{\varepsilon},x_2\right),\varepsilon\partial_{x_2}w_{i\nu}\left(\frac{x_1}{\varepsilon},x_2\right)\right),
		\end{equation*}
		and then,  
		$$
		\nabla^\varepsilon w_{i\nu}^\varepsilon(x_1,x_2) = \left(\partial_{x_1}w_{i\nu}\left(\frac{x_1}{\varepsilon},x_2\right),\partial_{x_2}w_{i\nu}\left(\frac{x_1}{\varepsilon},x_2\right)\right).
		$$
		
		Therefore, comparing $\nabla^\varepsilon w_{i\nu}^\varepsilon$ with $W_{i\nu}^\varepsilon$, we obtain 
		\begin{equation}\label{Wwinu}
		\nabla^\varepsilon w_{i\nu}^\varepsilon(x_1,x_2)=W_{i\nu}^\varepsilon(x_1,x_2)\mbox{ in }\Omega^\varepsilon.
		\end{equation}
		
		By \cite[Theorem 2.6]{cioranescu2000}, it follows from \eqref{aepsilon} that  
		\begin{eqnarray}\label{bsum00}
	\nonumber	a^\varepsilon_{i\nu}(\cdot,x_2)&\rightharpoonup & \frac{1}{L}\int_{0}^{L}a_{i\nu}(s,x_2)ds\\
		&=&a_p\left(<\partial_{x_1}w_0>_{i\nu}\right)\frac{1}{L}\int_0^L\chi_{Y^*}(s,x_2)a_p\left(\nabla P\phi(s,x_2)+(1,0)\right) ds
		\end{eqnarray} 
		weakly in $L^{p'}_{loc}(\mathbb{R})$ a.e. $x_2\in (0,g_1)$, when $\varepsilon\rightarrow 0$, 
		where $a_{i\nu}$ was defined in \eqref{ainu}.
		
		Moreover, if $\chi_{i\nu}$ is the characteristic function of the $\left(\frac{i}{2^\nu},\frac{(i+1)}{2^\nu}\right)$, the base of $A_{i\nu}$, we have that  
		\begin{equation} \label{bsum}
		\begin{gathered}
		\left(\frac{1}{L}\int_0^L\chi_{Y^*}(s,x_2)a_p\left(\nabla P\phi(s,x_2)+(1,0)\right) ds\right)\sum_{i\in I_\nu} a_p\left(<\partial_{x_1}w_0>_{i\nu}\right)\chi_{i\nu}(x_1)\\\rightharpoonup \left(\frac{1}{L}\int_0^L\chi_{Y^*}(s,x_2)a_p\left(\nabla P\phi(s,x_2)+(1,0)\right) ds\right)a_p(\partial_{x_1}w_0)
		\quad \textrm{ weakly in } L^{p'}(\Omega),
		\end{gathered}
		\end{equation}
		as $\nu\to \infty$.

		In fact, by Lebesgue Differentiation Theorem, we have that
		\begin{equation}\label{LDTaverage}
		\sum_{i\in I_\nu}<\partial_{x_1}w_0>_{i\nu}\chi_{i\nu}\rightarrow\partial_{x_1}w_0,  
		\end{equation}
		a.e. in $(0,1)$, where $I_\nu=\left\{0,1,\cdots, 2^\nu-1\right\}$. For the $L^p$ convergence, we just have to prove that there is a constant $K$ independent on $\nu$ such that 
		\begin{equation}\label{1684}
		\left|\left|\sum_{i\in I_\nu}<\partial_{x_1}w_0>_{i\nu}\chi_{i\nu}\right|\right|_{L^p(0,1)}\leq K.
		\end{equation}
		For all $x_1\in(0,1)$, there exists $j$ such that $x_1\in \left(\frac{j}{2^\nu},\frac{(j+1)}{2^\nu}\right)$. Then,
		\begin{equation}\label{1692}
			\sum_{i\in I_\nu}<\partial_{x_1}w_0>_{i\nu}\chi_{i\nu}(x_1)=2^\nu\int_{\frac{j}{2^\nu}}^{\frac{j+1}{2^\nu}}\partial_{x_1}w_0 dx_1.
		\end{equation}
		Thus, from Hardy-Littlewood-Wierner Theorem, 
		$$
		\int_0^1\left|\sum_{i\in I_\nu}<\partial_{x_1}w_0>_{i\nu}\chi_{i\nu}\right|^p dx_1\leq  c_p\|\partial_{x_1}w_0\|_{L^p(0,1)}^p,
		$$
		where $c_p>0$ is a constant that depends only on $p$, concluding the proof.

		Let us prove \eqref{bsum}. Using $j$ above, we get
		$$
		\begin{gathered}
		\left(\frac{1}{L}\int_0^L\chi_{Y^*}(s,x_2)a_p\left(\nabla P\phi(s,x_2)+(1,0)\right) ds\right)\sum_{i\in I_\nu} a_p(<\partial_{x_1}w_0>_{i\nu})\chi_{i\nu}(x_1)\\
		=\left(\frac{1}{L}\int_0^L\chi_{Y^*}(s,x_2)a_p\left(\nabla P\phi(s,x_2)+(1,0)\right) ds\right)a_p(<\partial_{x_1}w_0>_{j\nu}),\mbox{ a.e. }x_2\in(0,g_1).
		\end{gathered}
		$$
		Thus, by Corollary \ref{corolarioplaplaciano},
		\begin{equation*}
		\begin{gathered}
		\left|\left(\frac{1}{L}\int_0^L\chi_{Y^*}(s,x_2)a_p\left(\nabla P\phi(s,x_2)+(1,0)\right) ds\right)\sum_{i\in I_\nu} a_p(<\partial_{x_1}w_0>_{i\nu})\chi_{i\nu}(x_1)\right.\\\qquad\qquad\qquad -\left.\left(\frac{1}{L}\int_0^L\chi_{Y^*}(s,x_2)a_p\left(\nabla P\phi(s,x_2)+(1,0)\right) ds\right)a_p(\partial_{x_1}w_0)\right|\\
		\leq c \left|\left(\frac{1}{L}\int_0^L\chi_{Y^*}(s,x_2)a_p\left(\nabla P\phi(s,x_2)+(1,0)\right) ds\right)\right|\,\left|\sum_{i\in I_\nu} <\partial_{x_1}w_0>_{i\nu} \chi_{i\nu}(x_1) - \partial_{x_1}w_0(x_1)\right|^\alpha
		\rightarrow0,
		\end{gathered}
		\end{equation*}
		as  $\nu \to \infty$, for a.e. $(x_1,x_2)\in \Omega$, where $\alpha=\alpha(p)=\{1,p-1\}$. Notice also that, by \eqref{1684}, we get
		
		$$
		\begin{gathered}
		\left|\left(\frac{1}{L}\int_0^L\chi_{Y^*}(s,x_2)a_p\left(\nabla P\phi(s,x_2)+(1,0)\right) ds\right)\sum_{i\in I_\nu} a_p(<\partial_{x_1}w_0>_{i\nu})\chi_{i\nu}(x_1)\right|\\\leq c  \left|\frac{1}{L}\int_0^L\chi_{Y^*}(s,x_2)a_p\left(\nabla P\phi(s,x_2)+(1,0)\right) ds\right|,
		\end{gathered}
		$$
		where $c$ independs on $\nu$. Thus, it follows from Lebesgue Dominated Convergence Theorem that
		\begin{equation}\label{713}
		\begin{gathered}
		\int_{\Omega}\left(\frac{1}{L}\int_0^L\chi_{Y^*}(s,x_2)a_p\left(\nabla P\phi(s,x_2)+(1,0)\right) ds\right)\sum_{i\in I_\nu} a_p(<\partial_{x_1}w_0>_{i\nu})\chi_{i\nu}(x_1)\, \varphi dx_1dx_2\\
		\stackrel{\nu\to\infty}{\to} \,\int_{\Omega}\left(\frac{1}{L}\int_0^L\chi_{Y^*}(s,x_2)a_p\left(\nabla P\phi(s,x_2)+(1,0)\right) ds\right)a_p(\partial_{x_1}w_0) \varphi dx_1dx_2,
		\end{gathered}
		\end{equation}
		for any $\varphi\in L^{p}(\Omega)$, which implies \eqref{bsum}.

		\vspace{.5cm}
		
		\noindent\textbf{Auxiliar Problem}

		Take $\varphi\in W^{1,p}(\Omega^\varepsilon)$ with $\varphi=0$ in a neighborhood of the lateral boundaries of $\Omega^\varepsilon$. 
		Here we show that
		\begin{equation} \label{1104}
		\int_{\Omega^\varepsilon} \left|\nabla^\varepsilon w_{i\nu}^\varepsilon\right|^{p-2}\nabla^\varepsilon w_{i\nu}^\varepsilon\nabla^\varepsilon\varphi dx_1dx_2=0, \quad \forall \varepsilon>0. 
		\end{equation}
		Let $T^k_\varepsilon:Y^*_{k,\varepsilon}\rightarrow Y^*$ be the change of variables $T^k_\varepsilon(x_1,x_2)=\left(\frac{x_1-\varepsilon kL}{\varepsilon},x_2\right)$.  Here the set $Y^*_{k,\varepsilon}$ is such that 
		$$
		\Omega^\varepsilon=\bigcup_{k=0}^{N(\varepsilon)-1}Y^*_{k,\varepsilon}
		$$
		for some $N(\varepsilon) \in \mathbb{N}$. Next, define $T = T_{\varepsilon}^k$ for $(x_1,x_2) \in Y^*_{k,\epsilon}$.

		Notice that there are $m,n\in\{0,\dots,N(\varepsilon)-1\}$ such that $\textrm{supp }\varphi\subset \bigcup_{j=m}^n Y^*_{j,\varepsilon}$, that is, this family of $Y^*_{j,\varepsilon}$ is a finite covering of $\textrm{supp }\varphi$. Let $(\Theta_i)$, $i=m,\dots,n$ be a partition of unity associated to this covering. Then it satisfies
		$$
		\begin{gathered}
		\Theta_i\in C_0^\infty(\Omega^\varepsilon),\,0\leq\Theta_i\leq 1,\,\sum_{i=m}^{n}\Theta_i(x_1,x_2)=1,\, (x_1,x_2)\in \Omega^\varepsilon\\
		\textrm{supp }\Theta_i\subset Y^*_{i,\varepsilon}\, \quad\textrm{and}\quad \textrm{supp }\Theta_m\subset \Omega^\varepsilon\backslash\textrm{supp }\varphi.
		\end{gathered}
		$$
		Observe that 
		$$
		\varphi=\varphi\sum_{i=m}^{n}\Theta_i=\sum_{i=m+1}^{n}\varphi\Theta_i \quad\textrm{in}\quad \Omega^\varepsilon.
		$$
		because $\varphi=0$ in $\textrm{supp }\Theta_m\subset \Omega^\varepsilon\backslash\textrm{supp }\varphi$.
		
		Now, $\varphi\Theta_i$ is a function that is zero in a neighborhood of the boundaries of $Y^*_{i,\varepsilon}$, that is, $\varphi\Theta_i\circ (T_\varepsilon^k)^{-1}\in W^{1,p}_{per}(Y^*)$. Then,
		\begin{eqnarray*}
			&&\int_{\Omega^\varepsilon} \left|\nabla^\varepsilon w_{i\nu}^\varepsilon\right|^{p-2}\nabla^\varepsilon w_{i\nu}^\varepsilon\nabla^\varepsilon\varphi dx_1dx_2\\
			&=&\sum_{k=m+1}^{n}\int_{Y^*_{k,\varepsilon}} \left|\nabla^\varepsilon w_{i\nu}^\varepsilon\right|^{p-2}\nabla^\varepsilon w_{i\nu}^\varepsilon\nabla^\varepsilon(\varphi\Theta_k)  dx_1dx_2\\
			&=&\sum_{k=m+1}^{n}\frac{1}{\varepsilon^{p}}\int_{Y^*} |\nabla_y w_{i\nu}|^{p-2}\nabla_y w_{i\nu}\nabla_y\left((\varphi\Theta_k)\circ T^{-1}\right)\, \varepsilon \,dy_1dy_2\\
			&=&a_p\left(<\partial_{x_1}w_0>_{i\nu}\right)\sum_{k=m+1}^{n}\frac{1}{\varepsilon^{p-1}}\int_{Y^*} |\nabla_y v|^{p-2}\nabla_y v\nabla_y\left((\varphi\Theta_k)\circ \left(T_\varepsilon^k\right)^{-1}\right) dy_1dy_2 \\
			&=& 0.
		\end{eqnarray*}

		\vspace{.5cm}

		\noindent\textbf{Identifying the homogenized equation}

		We need to identify, for any $\varphi\in L^p(0,1)$,
		\begin{equation}
		\begin{gathered} \label{1281}
		\int_0^1\left[\int_0^{g_1}a_0(x_1,x_2)dx_2\right]( \varphi(x_1),0) dx_1 
		= \lim_{\varepsilon\rightarrow 0}\int_{\Omega}\left|\widetilde{\nabla^\varepsilon  u_\varepsilon}\right|^{p-2}\widetilde{\nabla^{\varepsilon} u_\varepsilon}
		\left(\varphi(x_1),0\right)dx_2dx_1
		\end{gathered}
		\end{equation}
		with 
		\begin{equation}
		\begin{gathered} \label{1288}
		\int_0^1 \left\{\int_0^{g_1}a_p\left[\partial_{x_1}u_0(x_1)\nabla P\phi(s,x_2)+(\partial_{x_1}u_0(x_1),0)\right] dx_2\right\}(\varphi(x_1),0) dx_1\\
		= \lim_{\nu\rightarrow \infty}\lim_{\varepsilon\rightarrow 0}\sum_{i\in I_\nu}\int_{A_{i\nu}} a^\varepsilon_{i\nu} 
		\left(\varphi(x_1),0\right)  \chi_{i\nu}(x_1)  \, dx_2dx_1.
		\end{gathered}
		\end{equation}
		Notice that \eqref{1281} and \eqref{1288} are obtained as consequence of \eqref{901}, \eqref{bsum00} and \eqref{bsum} taking $w_0 = u_0$. 
		
		Let $\eta_{i\nu}\in C_0^\infty\left(\frac{i}{2^\nu},\frac{i+1}{2^\nu}\right)$ with $0\leq \eta_{i\nu}\leq1$. 
		Take $\varphi =\eta_{i\nu} u_\varepsilon$ in \eqref{extendedvariationalproblem}. We get
		\begin{equation*}
		\int_{\Omega}a_p\left(\widetilde{\nabla^\varepsilon u_\varepsilon}\right)(\eta_{i\nu}' P_\varepsilon u_\varepsilon,0)dx_1dx_2+\int_{\Omega}\eta_{i\nu}a_p\left(\widetilde{\nabla^\varepsilon u_\varepsilon}\right)\nabla^\varepsilon u_\varepsilon+\chi_{\Omega^\varepsilon}a_p(P_\varepsilon u_\varepsilon)\eta_{i\nu} P_\varepsilon u_\varepsilon dx_1dx_2=\int_{\Omega}\chi_{\Omega^\varepsilon}f^\varepsilon \eta_{i\nu} P_\varepsilon u_\varepsilon dx_1dx_2
		\end{equation*}
		and for test functions $\varphi=\eta_{i\nu} w_{i\nu}^\varepsilon$
		\begin{equation*}
		\int_{\Omega}a_p\left(\widetilde{\nabla^\varepsilon u_\varepsilon}\right)(\eta_{i\nu}' w_{i\nu}^\varepsilon,0)dx_1dx_2+\int_{\Omega}\eta_{i\nu}a_p\left(\widetilde{\nabla^\varepsilon u_\varepsilon}\right)\nabla^\varepsilon w_{i\nu}^\varepsilon+\chi_{\Omega^\varepsilon}a_p(P_\varepsilon u_\varepsilon)\eta_{i\nu} w_{i\nu}^\varepsilon dx_1dx_2=\int_{\Omega}\chi_{\Omega^\varepsilon}f^\varepsilon \eta_{i\nu} w_{i\nu}^\varepsilon dx_1dx_2.
		\end{equation*}
		Thus, by \eqref{Pepsilonuepsilon}, \eqref{901} and \eqref{pseudolimit}, we get
		\begin{equation}\label{1818}
		\begin{gathered}
		\int_{\Omega}\eta_{i\nu}a_p\left(\widetilde{\nabla^\varepsilon u_\varepsilon}\right)\nabla^\varepsilon u_\varepsilon dx_1dx_2\to \int_0^1\eta_{i\nu} \hat{f} u dx_1-\int_\Omega \theta(x_2)a_{p}(u_0)\eta_{i\nu} u_0 dx_1dx_2-\int_{\Omega}a_0\cdot(\eta_{i\nu}'u_0,0)dx_1dx_2\\
		=\int_{\Omega}a_0\cdot(\partial_{x_{1}}(\eta_{i\nu}u_0),0)dx_1dx_2-\int_{\Omega}a_0\cdot(\partial_{x_{1}}\eta_{i\nu}u_0,0)dx_1dx_2=\int_{\Omega}\eta_{i\nu}a_0\cdot(\partial_{x_1}u_0,0)dx_1dx_2
		\end{gathered}
		\end{equation}
		and by \eqref{Pepsilonuepsilon}, \eqref{winuconv} and \eqref{pseudolimit}, we obtain
		\begin{equation}\label{1825}
		\begin{gathered}
		\int_{\Omega}\eta_{i\nu}a_p\left(\widetilde{\nabla^\varepsilon u_\varepsilon}\right)\nabla^\varepsilon w_{i\nu}^\varepsilon dx_1dx_2\to \int_{\Omega}\eta_{i\nu}a_0\cdot(<\partial_{x_1}w_0>_{i\nu},0)dx_1dx_2
		\end{gathered}
		\end{equation}
		as $\varepsilon\to 0$.
		
		On the other hand, using the definition of $a_{i\nu}^{\varepsilon}$ and taking as test functions $\varphi=\eta_{i\nu}P_{\varepsilon}u_\varepsilon-\eta_{i\nu}w_{i\nu}^\varepsilon$ in \eqref{1104}, we get
		$$
		0=\int_{\Omega}a_{i\nu}^{\varepsilon}\nabla^{\varepsilon}(\eta_{i\nu}P_{\varepsilon}u_\varepsilon-\eta_{i\nu}w_{i\nu}^\varepsilon)dx_{1}dx_{2}=\int_{\Omega^\varepsilon}a_p(\nabla^\varepsilon w_{i\nu}^\varepsilon)\nabla^\varepsilon(\eta_{i\nu}u_\varepsilon-\eta_{i\nu}w_{i\nu}^\varepsilon) dx_1dx_2, 
		$$
		which can be rewritten as follows
		\begin{equation*}
		\begin{gathered}
		\int_{\Omega^\varepsilon}\eta_{i\nu}a_{i\nu}^\varepsilon\nabla^\varepsilon( u_\varepsilon-w_{i\nu}^\varepsilon) dx_1dx_2
		=-\int_{\Omega}a_{i\nu}^\varepsilon\,(\eta_{i\nu}',0)(P_\varepsilon u_\varepsilon-w_{i\nu}^\varepsilon)dx_{1}dx_{2}\\\to-\int_{\Omega}\left[ a_p(<\partial_{x_1}w_0>_{i\nu})\dfrac{1}{L}\int_{0}^{L}\chi_{Y^*}(s,dx_2)a_p(\nabla P\phi(s,x_2)+(1,0))ds \right](\eta_{i\nu}',0)(u_0-<\partial_{x_1}w_0>_{i\nu}x_1)\,dx_1dx_2
		\end{gathered}
		\end{equation*}
		as $\varepsilon\to0$, since \eqref{bsum00}, \eqref{901} and \eqref{winuconv} hold. We rewrite the right hand side of above relation. Then,
		\begin{equation*}
		\begin{gathered}
		\int_{\Omega}\left[ a_p(<\partial_{x_1}w_0>_{i\nu})\dfrac{1}{L}\int_{0}^{L}\chi_{Y^*}(s,dx_2)a_p(\nabla P\phi(s,x_2)+(1,0))ds \right](\eta_{i\nu}',0)(u_0-<\partial_{x_1}w_0>_{i\nu}x_1)\,dx_1dx_2\\
		=\int_0^1 a_p(<\partial_{x_1}w_0>_{i\nu})\left[\dfrac{1}{L}\int_{0}^{L}\int_0^{g_1}\chi_{Y^*}(s,x_2)a_p(\nabla P\phi(s,x_2)+(1,0))ds dx_2\right](\eta_{i\nu}',0)(u_0-<\partial_{x_1}w_0>_{i\nu}x_1)\,dx_1\\
		=\dfrac{|Y^*|}{L}q\int_0^1 a_p(<\partial_{x_1}w_0>_{i\nu})\eta_{i\nu}'(u_0-<\partial_{x_1}w_0>_{i\nu}x_1)\,dx_1,
		\end{gathered}
		\end{equation*}
		where $q$ is defined by \eqref{defq}. Integrating by parts and using the fact that $\eta_{i\nu}\in C_0^\infty\left(\frac{i}{2^\nu},\frac{i+1}{2^\nu}\right)$, we get
		\begin{equation*}
		\begin{gathered}
		\int_{\Omega}\left[ a_p(<\partial_{x_1}w_0>_{i\nu})\dfrac{1}{L}\int_{0}^{L}\chi_{Y^*}(s,dx_2)a_p(\nabla P\phi(s,x_2)+(1,0))ds \right](\eta_{i\nu}',0)(u_0-<\partial_{x_1}w_0>_{i\nu}x_1)\,dx_1dx_2\\
		=-\dfrac{|Y^*|}{L}q\int_0^1 a_p(<\partial_{x_1}w_0>_{i\nu})\eta_{i\nu}(\partial_{x_1}u_0-<\partial_{x_1}w_0>_{i\nu})\,dx_1.
		\end{gathered}
		\end{equation*}
		This last equality leads us to 
		\begin{equation}\label{1859}
		\int_{\Omega}\eta_{i\nu}a_{i\nu}^\varepsilon\nabla^\varepsilon(u_\varepsilon-w_{i\nu}^\varepsilon) dx_1dx_2\to\dfrac{|Y^*|}{L}q\int_0^1 a_p(<\partial_{x_1}w_0>_{i\nu})\eta_{i\nu}(\partial_{x_1}u_0-<\partial_{x_1}w_0>_{i\nu})\,dx_1
		\end{equation}
		as $\varepsilon\to 0$.
		
		Notice that, by \eqref{aepsilon}, \eqref{Wwinu} and monotonicity, we have 
		\begin{equation}\label{2000}
		\begin{gathered}
		0\leq \sum_{i=0}^{2^\nu-1} \int_{A_{i\nu}}\eta_{i\nu}\left(\left|\widetilde{\nabla^\varepsilon  u_\varepsilon}\right|^{p-2}\widetilde{\nabla^{\varepsilon} u_\varepsilon}-a_{i\nu}^\varepsilon\right)\left(\widetilde{\nabla^\varepsilon u_\varepsilon} -W_{i\nu}^\varepsilon\right)dx_1dx_2\\
		=\sum_{i=0}^{2^\nu-1}\left[\int_{A_{i\nu}}\eta_{i\nu}a_p(\widetilde{\nabla^\varepsilon  u_\varepsilon}) \left(\widetilde{\nabla^\varepsilon u_\varepsilon} -\nabla^\varepsilon w_{i\nu}^\varepsilon\right) dx_1dx_2-\int_{A_{i\nu}\cap\Omega^\varepsilon}\eta_{i\nu}a_{i\nu}^\varepsilon\left(\nabla^\varepsilon u_\varepsilon -\nabla^\varepsilon w_{i\nu}^\varepsilon\right)dx_1dx_2\right].
		\end{gathered}
		\end{equation}
		
		Passing to the limit, as $\varepsilon\to 0$, we get, by \eqref{1818}, \eqref{1825} and \eqref{1859}, that
		\begin{equation}
		\begin{gathered}
		0\leq \sum_{i=0}^{2^\nu-1} \int_{A_{i\nu}}\eta_{i\nu}\left(\left|\widetilde{\nabla^\varepsilon  u_\varepsilon}\right|^{p-2}\widetilde{\nabla^{\varepsilon} u_\varepsilon}-a_{i\nu}^\varepsilon\right)\left(\widetilde{\nabla^\varepsilon u_\varepsilon} -W_{i\nu}^\varepsilon\right)dx_1dx_2\\
		\to \sum_{i=0}^{2^\nu-1}\left\lbrace\int_{A_{i\nu}}\eta_{i\nu}a_0\cdot(\partial_{x_1}u_0,0)dx_1dx_2-\int_{A_{i\nu}}\eta_{i\nu}a_0\cdot(<\partial_{x_1}w_0>_{i\nu},0)dx_1dx_2\right.\\
		\left.-\dfrac{|Y^*|}{L}q\int_0^1 a_p(<\partial_{x_1}w_0>_{i\nu})\eta_{i\nu}(\partial_{x_1}u_0-<\partial_{x_1}w_0>_{i\nu})\,dx_1\right\rbrace
		\end{gathered}
		\end{equation}
		and then $\eta_{i\nu}\to 1$, 
		\begin{equation}
		\begin{gathered}
		0\leq \sum_{i=0}^{2^\nu-1}\left\lbrace\int_{A_{i\nu}}a_0\cdot(\partial_{x_1}u_0,0)dx_1dx_2-\int_{A_{i\nu}} a_0\cdot(<\partial_{x_1}w_0>_{i\nu},0)dx_1dx_2\right.\\
		\left.-\dfrac{|Y^*|}{L}q\int_0^1 a_p(<\partial_{x_1}w_0>_{i\nu})(\partial_{x_1}u_0-<\partial_{x_1}w_0>_{i\nu})\,dx_1\right\rbrace
		\end{gathered}
		\end{equation}
		folllowed by $\nu\to\infty$, we have
		\begin{equation}\label{1981}
		0\leq \int_{\Omega}a_0(\partial_{x_1}u_0-\partial_{x_1}w_0,0) dx_1dx_2-\dfrac{|Y^*|}{L}q\int_0^1 a_p(\partial_{x_1}w_0)(\partial_{x_1}u_0-\partial_{x_1}w_0)\,dx_1,
		\end{equation}
		where this last convergence is obtained by \eqref{713} and \eqref{LDTaverage}.

		Now, by \eqref{prehomogenized}, we get 
		\begin{equation}\label{quasiend}
		0\leq - \frac{|Y^*|}{L} \int_0^1a_p(u_0) (u_0-w_0) dx_1   +\int_0^1  \hat f(x_1) (u_0-w_0) dx_1-\dfrac{|Y^*|}{L}q\int_0^1 a_p(\partial_{x_1}w_0)(\partial_{x_1}u_0-\partial_{x_1}w_0)\,dx_1.
		\end{equation}
		Since $w_0\in W^{1,p}(0,1)$ is an arbitrary function,  let us take $w_0=u_0-\lambda\psi$ in \eqref{quasiend}, 
		with $\lambda>0$. If we divide \eqref{quasiend} by $\lambda$,  we get that 
		\begin{eqnarray}
		\nonumber&&\dfrac{|Y^*|}{L}q\int_0^1 a_p(\partial_{x_1}u_0-\lambda \partial_{x_1}\psi)(\partial_{x_1}u_0-\partial_{x_1}w_0)\,dx_1\leq \int_0^1\left(\hat{f}-\frac{|Y^*|}{L}a_p(u_0)\right)\psi dx_1.
		\end{eqnarray}
		On another hand, if we take $w_0=u_0+\lambda\psi$ with $\lambda>0$,  we obtain
		\begin{eqnarray}
		\nonumber&&\dfrac{|Y^*|}{L}q\int_0^1 a_p(\partial_{x_1}u_0-\lambda \partial_{x_1}\psi)\partial_{x_1}\psi\,dx_1\geq \int_0^1\left(\hat{f}-\frac{|Y^*|}{L}a_p(u_0)\right)\psi dx_1.
		\end{eqnarray}
		Therefore, making $\lambda\rightarrow 0$ in the previous inequalities, we get that 
		\begin{eqnarray}\label{problemahomogenizadofinal}
		\dfrac{|Y^*|}{L}q\int_0^1 a_p(\partial_{x_1}u_0)\partial_{x_1}\psi\,dx_1= \int_0^1\left(\hat{f}-\frac{|Y^*|}{L}a_p(u_0)\right)\psi dx_1.
		\end{eqnarray} 
		
		Rewriting \eqref{problemahomogenizadofinal}, we get
		\begin{equation*}  
		q\int_0^1 \,|\partial_{x_1}u_0|^{p-2}\partial_{x_1}u_0 \partial_{x_1}\psi dx_1+\int_{0}^1 |u_0|^{p-2}u_0 \psi dx_1=\int_0^1\overline{f}\psi dx_1,
		\end{equation*}
		for all $\psi\in W^{1,p}(0,1)$, where 
		\begin{equation*}
		\overline{f}=\frac{L}{|Y^*|}\hat{f}.
		\end{equation*}

	\end{proof}

	\section{Corrector Result}
	
	In this section we introduce a corrector to the problem \eqref{problemchanged}.
	According to \cite{MasoA}, since we already have 
	$$P_\varepsilon u_\varepsilon \rightarrow u_0, \quad \textrm{ strongly in } L^p(\Omega),$$
	we just need to construct the corrector to the term $\nabla^\varepsilon u_\varepsilon$.  
	
	For this sake, consider the partition $\left\{A_{i\nu}\right\}$ introduced in Section 3 and defined by \eqref{Anu}. 
	Let $M_\nu$ be the following family of functions 
	\begin{equation*}
	M_\nu\varphi(x_{1})=\sum_{k=0}^{2^\nu-1}\chi_{i\nu}(x_{1})<\varphi>_{i\nu}, \quad \varphi\in L^p(\Omega^\varepsilon).
	\end{equation*}
	Arguing as in \eqref{bsum}, we can show that 
	$$M_\nu\varphi\rightarrow \varphi \quad \textrm{ for all } \varphi \in L^p(\Omega^\varepsilon)$$
	and then, we can say that $M_\nu$ is an approximation to the identity map in $L^p(\Omega^\varepsilon)$. 
	
	Now, let $\phi$ be such that $v=\phi+y_1$ is the solution of the auxiliar problem \eqref{auxint}. Extend $P\phi$ periodically in the first variable and consider 
	\begin{equation*}
	W_{i\nu}^\varepsilon(x_1,x_2)=< \partial_{x_1}u_0>_{i\nu}\left[\nabla P\phi\left(\frac{x_1}{\varepsilon},x_2\right)+\left(1,0\right)\right]\textrm{ in }\Omega.
	\end{equation*}
	It follows from \eqref{winuconv} and \eqref{Wwinu} that 
	$$
	W_{i\nu}^\varepsilon\stackrel{\varepsilon\rightarrow 0}{\rightharpoonup}\left(< \partial_{x_1}u_0>_{i\nu},0\right) 
	\quad \textrm{ weakly in } L^p(\Omega) \times L^p(\Omega).
	$$
	
	Here, we combine $M_\nu$ and $W_{i\nu}^\varepsilon$ to introduce our corrector function by the expression 
	\begin{equation}\label{corrector}
	c_{i\nu}^\varepsilon(x_1,x_2)=\sum_{k=0}^{2^\nu-1} \chi_{i\nu}(x_1)W_{i\nu}^\varepsilon(x_1,x_2)\textrm{ in }\Omega.
	\end{equation}
	
	First, we see that our corrector weakly approximate $\nabla u_0$. Indeed, for all $\varphi\in L^{p'}(\Omega)^2$, we have 
	\begin{eqnarray*} 
	\int_{\Omega}c_{i\nu}^\varepsilon\varphi & = & \sum_{i=0}^{2^\nu-1}\int_{A_{i\nu}} W_{i\nu}^\varepsilon\varphi 
	 \stackrel{\varepsilon\rightarrow 0}{\rightarrow}  \sum_{i=0}^{2^\nu-1}\int_{A_{i\nu}} \left(< \partial_{x_1}u_0>_{i\nu},0\right)\varphi \\
	& \stackrel{\nu\rightarrow\infty}{\rightarrow} & \int_{\Omega}\nabla u_0\varphi.
	\end{eqnarray*}

	Next, let us show the strong convergence to $\nabla u_0$. For that, let us take $\eta\in C_{0}^{\infty}(0,1)$ with $0\leq \eta \leq 1$, and set the following notation to simplify our arguments: $d\mu=\eta dx_1dx_2$.
	By Proposition \ref{proposicaoplaplaciano}, if $p\geq 2$,
	\begin{equation}\label{corretorineqp>2}
		\int_{\Omega^\varepsilon}\left|c_{i\nu}^\varepsilon-\nabla^\varepsilon u_\varepsilon\right|^pd\mu\leq c\int_{\Omega^\varepsilon} \left(\left|c_{i\nu}^\varepsilon\right|^{p-2}c_{i\nu}^\varepsilon-\left|\nabla^\varepsilon u_\varepsilon\right|^{p-2}\nabla^\varepsilon u_\varepsilon\right)\left(c_{i\nu}^\varepsilon-\nabla^\varepsilon u_\varepsilon\right)d\mu.
	\end{equation}
	
	For $1<p<2$, it follows from H\"older's Inequality and Proposition \ref{proposicaoplaplaciano} that 
	\begin{eqnarray}\label{correctorineqp<2}
		&&\int_{\Omega^\varepsilon} \nonumber \left|c_{i\nu}^\varepsilon-\nabla^\varepsilon u_\varepsilon\right|^pd\mu\\\nonumber&=&\int_{\Omega^\varepsilon}\left|c_{i\nu}^\varepsilon-\nabla^\varepsilon u_\varepsilon\right|^p\frac{\left(\left|c_{i\nu}^\varepsilon\right|+\left|\nabla^\varepsilon u_\varepsilon\right|\right)^{\frac{(p-2)p}{2}}}{\left(\left|c_{i\nu}^\varepsilon\right|+\left|\nabla^\varepsilon u_\varepsilon\right|\right)^{\frac{(p-2)p}{2}}}d\mu\\
	\nonumber&\leq&\left(\int_{\Omega^\varepsilon}\left|c_{i\nu}^\varepsilon-\nabla^\varepsilon u_\varepsilon\right|^2 \left(\left|c_{i\nu}^\varepsilon\right|+\left|\nabla^\varepsilon u_\varepsilon\right|\right)^{p-2}d\mu\right)^{\frac{p}{2}} \left(\int_{\Omega^\varepsilon}\left(\left|c_{i\nu}^\varepsilon\right|+\left|\nabla^\varepsilon u_\varepsilon\right|\right)^{p}d\mu\right)^{\frac{2-p}{2}}\\
		\nonumber&\leq&c\left(\int_{\Omega^\varepsilon}\left(
		\left|c_{i\nu}^\varepsilon\right|^{p-2}c_{i\nu}^\varepsilon
		-\left|\nabla u_\varepsilon\right|^{p-2}\nabla u_\varepsilon\right)\left(c_{i\nu}^\varepsilon-\nabla^\varepsilon u_\varepsilon\right)d\mu\right)^{\frac{p}{2}} \left(\int_{\Omega^\varepsilon}\left(\left|c_{i\nu}^\varepsilon\right|+\left|\nabla^\varepsilon u_\varepsilon\right|\right)^{p}d\mu\right)^{\frac{2-p}{2}}\\
		&\leq&c\left(\int_{\Omega^\varepsilon}\left(\left|c_{i\nu}^\varepsilon\right|^{p-2}c_{i\nu}^\varepsilon-\left|\nabla^\varepsilon u_\varepsilon\right|^{p-2}\nabla^\varepsilon u_\varepsilon\right)\left(c_{i\nu}^\varepsilon- \nabla^\varepsilon u_\varepsilon \right)d\mu\right)^{\frac{p}{2}}
	\end{eqnarray}
	since $c_{i\nu}^\varepsilon$ and $\nabla^\varepsilon u_\varepsilon$ are uniformly bounded.
	
	In this way, one just need to pass to the limit, as $\varepsilon\to0$, to the term
	\begin{equation}\label{almostcorrector}
	\begin{gathered}
	\int_{\Omega^\varepsilon}\left(\left|c_{i\nu}^\varepsilon\right|^{p-2}c_{i\nu}^\varepsilon-\left|\nabla^\varepsilon u_\varepsilon\right|^{p-2}\nabla^\varepsilon u_\varepsilon\right)\left(c_{i\nu}^\varepsilon- \nabla^\varepsilon u_\varepsilon \right)d\mu\\=\int_{\Omega^\varepsilon}\eta\left(\left|c_{i\nu}^\varepsilon\right|^{p-2}c_{i\nu}^\varepsilon-\left|\nabla^\varepsilon u_\varepsilon\right|^{p-2}\nabla^\varepsilon u_\varepsilon\right)\left(c_{i\nu}^\varepsilon- \nabla^\varepsilon u_\varepsilon \right)dx_1dx_2.
	\end{gathered}
	\end{equation}
	In fact, since   
		\begin{equation} \label{equalitiescorrector}
		\begin{gathered}
		\int_{\Omega^\varepsilon}\eta\left|c_{i\nu}^\varepsilon\right|^{p-2}c_{i\nu}^\varepsilon c_{i\nu}^\varepsilon dx_1dx_2
		=\sum_{i=0}^{2^\nu-1}\int_{\Omega^\varepsilon\cap A_{i\nu}}\eta\left|W_{i\nu}^\varepsilon\right|^{p-2}W_{i\nu}^\varepsilon W_{i\nu}^\varepsilon  dx_1dx_2\\
		\int_{\Omega^\varepsilon}\eta\left|c_{i\nu}^\varepsilon\right|^{p-2}c_{i\nu}^\varepsilon \nabla^\varepsilon u_\varepsilon dx_1dx_2
		=\sum_{i=0}^{2^\nu-1}\int_{\Omega^\varepsilon\cap A_{i\nu}}\eta\left|W_{i\nu}^\varepsilon\right|^{p-2}W_{i\nu}^\varepsilon \nabla^\varepsilon u_\varepsilon  dx_1dx_2\\
		\int_{\Omega^\varepsilon}\eta\left|\nabla^\varepsilon u_\varepsilon\right|^{p-2}\nabla^\varepsilon u_\varepsilon c_{i\nu}^\varepsilon dx_1dx_2
		=\sum_{i=0}^{2^\nu-1}\int_{\Omega^\varepsilon\cap A_{i\nu}}\eta\left|\nabla^\varepsilon u_\varepsilon\right|^{p-2}\nabla^\varepsilon u_\varepsilon W_{i\nu}^\varepsilon  dx_1dx_2,
	\end{gathered}
	\end{equation}
	we can put together \eqref{almostcorrector} and \eqref{equalitiescorrector} to pass to the limit in 
	\begin{equation}\label{almostcorrector01}
	\begin{gathered}
	\int_{\Omega^\varepsilon}\left(\left|c_{i\nu}^\varepsilon\right|^{p-2}c_{i\nu}^\varepsilon-\left|\nabla^\varepsilon u_\varepsilon\right|^{p-2}\nabla^\varepsilon u_\varepsilon\right)\left(c_{i\nu}^\varepsilon- \nabla^\varepsilon u_\varepsilon \right)d\mu \\ 
	= \sum_{i=0}^{2^\nu-1}\int_{\Omega^\varepsilon\cap A_{i\nu}}\eta\left(\left|W_{i\nu}^\varepsilon\right|^{p-2}W_{i\nu}^\varepsilon - \left|\nabla^\varepsilon u_\varepsilon\right|^{p-2}\nabla^\varepsilon u_\varepsilon\right)(W_{i\nu}^\varepsilon -\nabla^\varepsilon u_\varepsilon)dx_{1}dx_{2}.
	\end{gathered}
	\end{equation}
	It envolves the same arguments as we did to pass to the limit in \eqref{2000} in order to obtain \eqref{1981}.

	Indeed, taking in account \eqref{Wwinu}, \eqref{almostcorrector}, \eqref{almostcorrector01}, \eqref{corretorineqp>2} and \eqref{correctorineqp<2}, we conclude  
	\begin{equation*}
	\lim_{\nu\rightarrow\infty}\lim_{\varepsilon\rightarrow 0}\int_{\Omega^\varepsilon}\left|c_{i\nu}^\varepsilon-\nabla^\varepsilon u_\varepsilon\right|^pd\mu=0.
	\end{equation*} 
	
	Therefore, we can state the following corrector result concerning to the asymptotic behavior of $u_\varepsilon$:
	\begin{prop}
		Let $(u_\varepsilon)$ be the sequence of solutions given by problem \eqref{problem}, and $(c_{i\nu}^\varepsilon)$ be the functions defined by \eqref{corrector}. 
		Then, 
		\begin{equation*}
		\lim_{\nu\rightarrow\infty}\lim_{\varepsilon\rightarrow 0}\left|\left|c_{i\nu}^\varepsilon-\nabla^\varepsilon u_\varepsilon\right|\right|_{L^{p}(\Omega^\varepsilon)}=0.
		\end{equation*} 
	\end{prop}
	
	\vspace{0.3 cm}

	{\bf Acknowledgements.} 
	
	The first author (MCP) is partially supported by CNPq 303253/2017-7 and FAPESP 2017/02630-2 (Brazil) and the second one (JCN) by CNPq Scholarship 141675/2015-2 (Brazil). 
	
	The authors would like to thank the anonymous referees whose comments have considerably improved the writing of the paper.

\end{document}